\documentclass[11pt,reqno]{amsart}
\usepackage[utf8]{inputenc}
\usepackage[x11names]{xcolor}
\usepackage{geometry}
\usepackage{amsmath,amsfonts,amssymb,amsthm,mathtools}
\usepackage{float}
\usepackage{tikz}
\usepackage{tikz-cd}
\usetikzlibrary{babel}
\usetikzlibrary{arrows.meta,decorations.markings}
\definecolor{graphblue}{RGB}{30,91,146}
\definecolor{graphorange}{RGB}{182,78,17}
\allowdisplaybreaks[1]

\usepackage{hyperref}
\hypersetup{
    colorlinks=true,
    linktoc=all,
    linkcolor=magenta,
    linktocpage,
    citecolor=DodgerBlue4
}

\newcommand{\Fp}{\mathbb F_p}
\newcommand{\Cp}{C_p}
\newcommand{\E}{\mathbb{E}}
\newcommand{\THH}{\operatorname{THH}}
\newcommand{\TR}{\operatorname{TR}}

\newtheorem{theorem}{Theorem}
\newtheorem{lemma}[theorem]{Lemma}
\theoremstyle{definition}

\begin{document}

\title{A short proof of the Segal conjecture for $C_p$}
\author{Thomas Nikolaus}
\date{\today}
\maketitle

Let $p$ be a prime and $C_p = \{1,\sigma, \ldots, \sigma^{p-1}\}$ the cyclic group of order $p$. We give a new and direct proof of the following result.
The version for finite spectra was proved by Jones and Wegmann
\cite[Theorem A(ii) and the result on page 474 for odd $p$]{JW}.
Lun{\o}e-Nielsen and Rognes proved the result for bounded below spectra
whose mod-$p$ homology is of finite type \cite[Theorem 5.13]{LNR}.
The statement in the generality given here is
\cite[Theorem III.1.7]{NS}.

\begin{theorem}\label{thm:main} 
For every bounded below spectrum $X$, the Tate diagonal
\[
X \longrightarrow (X^{\otimes p})^{tC_p},
\]
is a $p$-completion. Here  $C_p$ acts by cyclic permutation on $X^{\otimes p} = X \otimes \ldots \otimes X$.
\end{theorem}

For $X$ the sphere spectrum this recovers the Segal conjecture
for $C_p$, which was proved by Lin for $p=2$ \cite{Lin}
and by Gunawardena for odd $p$ \cite{Gun}.
The proof given in \cite{NS} is a formal reduction of Theorem \ref{thm:main}
to the special case $X=\Fp$, which then follows by 
Lun{\o}e-Nielsen and Rognes' finite-type theorem since $\Fp$ is of finite type.

Here we give a direct proof of the case $X=\Fp$ and hence, by this reduction, of the full theorem.
Our proof is inspired by Hahn and Wilson's descent argument for
$p=2$ \cite{HW} and its extension to odd primes by Angelini-Knoll
and Zou using spoke topological Hochschild homology
\cite[Sections~5--6]{AKZ}.
We however use a different descent here and do not need any computational input except the  well-known calculation
\begin{equation}\label{eq:input}
 \pi_*\THH(\Fp)=\Fp[u],\qquad
 \pi_*\THH(\Fp)^{t\Cp}=\Fp[x^{\pm1}].
\end{equation}
Here \(|u|=|x|=2\), and the cyclotomic Frobenius sends \(u\) to \(x\).
The first isomorphism in \eqref{eq:input} is due to B\"okstedt
\cite[Theorem~1.1(a), (d)]{Bok}; the Tate computation and the
identification of the cyclotomic Frobenius are due to
Hesselholt--Madsen \cite[equation~(33) and Proposition~5.3]{HM}.
For a direct proof of B\"okstedt periodicity from the dual Steenrod algebra
and its Dyer--Lashof operations, see Krause--Nikolaus
\cite[Theorems~1.1--1.2]{KN}. 
The statements \eqref{eq:input} can also be deduced entirely without any computation using polynomial functoriality and $\TR$ as in
\cite[Theorem~5.5 and Corollary~5.14]{BGMN}. \\

In order to prove the theorem we use the $C_p$-equivariant map of $\E_\infty$-rings
\[
 A_{-1} = \Fp^{\otimes p}\longrightarrow\THH(\Fp) = A_0
\]
induced by the inclusion of $p$  points into the circle (equivalently by the inclusion of the term in simplicial degree $p-1$ into the realization of the cyclic bar construction or by the universal property of the source as the free $\E_\infty$-algebra  under $\Fp$ with $C_p$-action).
We consider the Amitsur complex of this map with $i$-th term
$A_i:=A_0\otimes_{A_{-1}}\cdots\otimes_{A_{-1}}A_0$, with $i+1$ factors.
This is an augmented cosimplicial $\E_\infty$-ring with $C_p$-action:
\begin{equation}\label{bare_diagram}
\begin{tikzcd}[column sep=large]
A_{-1}
  \arrow[r, "\eta"]
&
A_0
  \arrow[r, shift left=0.5ex, "d^0"]
  \arrow[r, shift right=0.5ex, "d^1"']
&
A_1
  \arrow[r, shift left=1ex, "d^0"]
  \arrow[r, "d^1" description]
  \arrow[r, shift right=1ex, "d^2"']
&
A_2
  \arrow[r, shift left=1.5ex]
  \arrow[r, shift left=0.5ex]
  \arrow[r, shift right=0.5ex]
  \arrow[r, shift right=1.5ex]
&
\cdots
\end{tikzcd}
\end{equation}

\begin{lemma}\label{lem:tate-descent}
The diagram
\[
\begin{tikzcd}[column sep=large]
A_{-1}^{tC_p}
  \arrow[r, "\eta"]
&
A_0^{tC_p}
  \arrow[r, shift left=0.5ex, "d^0"]
  \arrow[r, shift right=0.5ex, "d^1"']
&
A_1^{tC_p}
  \arrow[r, shift left=1ex, "d^0"]
  \arrow[r, "d^1" description]
  \arrow[r, shift right=1ex, "d^2"']
&
A_2^{tC_p}
  \arrow[r, shift left=1.5ex]
  \arrow[r, shift left=0.5ex]
  \arrow[r, shift right=0.5ex]
  \arrow[r, shift right=1.5ex]
&
\cdots
\end{tikzcd}
\]
is a limit diagram.
\end{lemma}
\begin{proof}
The fibre $I$ of $A_{-1}\to A_0$ is $1$-connective.
The iterated fibres in the associated Tot-tower  are tensor
powers of $I$ over $A_{-1}$, so their connectivity tends to infinity.
Thus \eqref{bare_diagram} is a limit diagram; see also \cite{DR}.
Applying homotopy fixed points preserves this limit.
Applying homotopy orbits also gives a limit diagram, since
homotopy orbits preserve finite limits and connectivity.
The claim then follows since Tate is the cofibre of the norm map from
homotopy orbits to homotopy fixed points.
\end{proof}

The $i+1$ equivariant maps $A_0\to A_i$ induce after taking the Tate construction classes
$x_0,\ldots,x_i$ in $\pi_2(A_i^{tC_p})$, each the image of $x \in \pi_2(A_0^{tC_p})$.
Together these classes give rise to a map
\begin{equation}\label{eq:map}
 \Fp[x_0^{\pm1},\ldots,x_i^{\pm1}]
 \longrightarrow\pi_*(A_i^{tC_p}),
 \qquad |x_j|=2.
\end{equation}
As $i$ varies, these maps form a morphism of augmented cosimplicial graded
commutative rings. The source is augmented by $\Fp$, and the map
in degree $-1$ is induced by the Tate diagonal
$\Fp\to A_{-1}^{tC_p}$. The cosimplicial maps in the source act
on the indices of the generators.
Consider the ideal
\[
 I_i=(x_s-x_t\mid 0\leq s,t\leq i)
 \subseteq \Fp[x_0^{\pm1},\ldots,x_i^{\pm1}].
\]

\begin{lemma}\label{prop:completion}
For $i \geq 0$ the map \eqref{eq:map} exhibits $\pi_*(A_i^{tC_p})$ as the degreewise $I_i$-adic completion of $\Fp[x_0^{\pm1},\ldots,x_i^{\pm1}]$.
\end{lemma}

The completion is taken in the graded sense. Set
$z_j:=(x_j-x_0)/x_0$ for $j=1,\ldots,i$. Then
\begin{equation}\label{identification}
 C^i:=\Fp[x_0^{\pm1},\ldots,x_i^{\pm1}]^\wedge_{I_i}
 =\Fp[\![z_1,\ldots,z_i]\!][x_0^{\pm1}],
 \qquad |z_j|=0.
\end{equation}
i.e. a 2-periodic version of  the power series ring $\Fp[\![z_1,\ldots,z_i]\!]$. These coordinates $z_j$ of course depend on the
choice of $x_0$ and the cosimplicial maps are a bit messy in this description.
\begin{proof}
The $\E_\infty$-ring $\THH(\Fp)$ is obtained by tensoring
$\Fp$ with $S^1$ in the category of $\E_\infty$-rings, while $\Fp^{\otimes p}$ is obtained by
tensoring with the free $C_p$-set of $p$ points, embedded in $S^1$
as the $p$th roots of unity. Since tensoring preserves colimits,
$A_i$ is obtained by tensoring $\Fp$ with the $C_p$-space $T_i$
formed by gluing $i+1$ circles along these marked points.
The quotient $T_i/C_p$ is a wedge of $i+1$ circles.
Choosing the zeroth circle as the base circle gives the following picture:

\begin{figure}[H]
\centering
\begin{tikzpicture}[x=1cm,y=1cm,font=\small,
 base/.style={draw=graphblue,line width=1.2pt},
 first/.style={draw=graphorange,line width=1pt,densely dashed},
 second/.style={draw=teal!75!black,line width=1pt},
 third/.style={draw=violet!80!black,line width=1.1pt,densely dotted},
 directed/.style={postaction={decorate},decoration={markings,
  mark=at position 0.55 with {\arrow{Stealth[length=1.6mm,width=0.9mm]}}}},
 point/.style={circle,fill=black,inner sep=1.8pt}]

\begin{scope}
\foreach \a/\b in {90/210,210/330,330/90}{
 \draw[base,directed] (\a:2) -- (\b:2);
 \draw[first,directed] (\a:2) to[bend right=20] (\b:2);
 \draw[second,directed] (\a:2) to[bend right=40] (\b:2);
 \draw[third,directed] (\a:2) to[bend right=60] (\b:2);
}
\foreach \a in {90,210,330}{
 \node[point] at (\a:2) {};
}
\node at (0,0) {$T_3$};
\end{scope}
\end{tikzpicture}
\caption{The space $T_3$ for $p=3$: four circles glued at three marked
points. The blue triangle is the chosen base circle. The other
circles are orange, green and purple. The $C_3$ action is given by rotation.}
\label{fig:three-points-four-circles}
\end{figure}
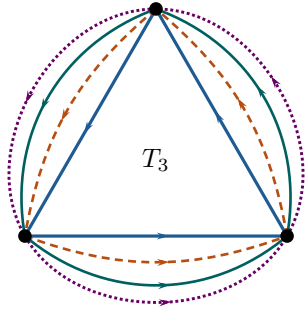

The space $T_i$ is equivariantly homotopy equivalent to the base circle
with $i$ free $C_p$-orbits of additional loops attached at the marked points:

\begin{figure}[H]
\centering
\begin{tikzpicture}[x=1cm,y=1cm,font=\small,
 base/.style={draw=graphblue,line width=1.2pt},
 first/.style={draw=graphorange,line width=1pt,densely dashed},
 second/.style={draw=teal!75!black,line width=1pt},
 third/.style={draw=violet!80!black,line width=1.1pt,densely dotted},
 directed/.style={postaction={decorate},decoration={markings,
  mark=at position 0.55 with {\arrow{Stealth[length=1.6mm,width=0.9mm]}}}},
 point/.style={circle,fill=black,inner sep=1.8pt}]

\foreach \a/\b in {90/210,210/330,330/90}{
 \draw[base,directed] (\a:2) -- (\b:2);
}

\foreach \a in {90,210,330}{
 \begin{scope}[shift={(\a:2)},rotate=\a]
  \foreach \angle/\loopstyle in {-45/first,0/second,45/third}{
   \begin{scope}[rotate=\angle]
    \draw[\loopstyle,directed] (0,0)
     .. controls (1.35,-0.45) and (1.35,0.45) .. (0,0);
   \end{scope}
  }
 \end{scope}
}
\foreach \a in {90,210,330}{
 \node[point] at (\a:2) {};
}
\end{tikzpicture}
\caption{The homotopy equivalent space for $p=3$, $i=3$: the blue base circle
with three loops at each marked point. Each colour forms a free
$C_3$-orbit under rotation by $120^\circ$.
The homotopy equivalence to the graph $T_i$ of Figure~\ref{fig:three-points-four-circles}
fixes the blue circle and sends each coloured loop to the
corresponding coloured edge followed by the blue edge backwards.}
\label{fig:equivariant-edge-slide}
\end{figure}
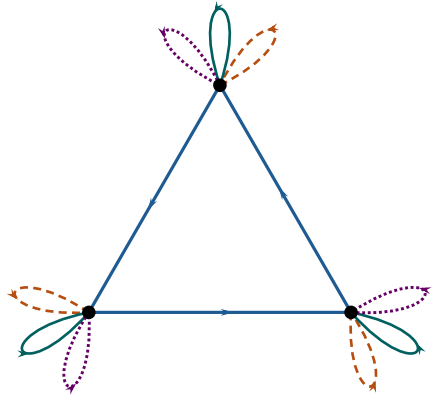

Put
$
 P_i=\THH(\Fp)^{\otimes_{\Fp}i}$ and $
 N(P_i)=P_i^{\otimes p},
$
where $C_p$ cyclically permutes the factors of $N(P_i)$. From the pushout of spaces we get
\begin{equation}\label{basechange}
 A_i\simeq A_0\otimes_{A_{-1}}N(P_i).
\end{equation}
Let $u_1,\ldots,u_i$ be the B\"okstedt classes of the attached loops at one of the vertices.
Then
$
 \pi_*P_i=\Fp[u_1,\ldots,u_i]$
with $|u_j| = 2$ for all $j$.
Since the tensor products defining $N(P_i)$ are taken over the sphere,
we get
\[
 \pi_*N(P_i)=
 (\pi_*A_{-1})[u_1,\ldots,u_i,\sigma u_1,\ldots,\sigma u_i,
              \ldots,\sigma^{p-1}u_1,\ldots,\sigma^{p-1}u_i],
\]
where $\sigma$ is a generator of $C_p$. More precisely, $N(P_i)$
decomposes as a direct sum of suspensions of $A_{-1}$ indexed by
monomials. By \eqref{basechange}, we obtain
\[
\pi_*(A_i) = \pi_*(A_0) [u_1,\ldots,u_i,\sigma u_1,\ldots,\sigma u_i,
              \ldots,\sigma^{p-1}u_1,\ldots,\sigma^{p-1}u_i].
\]
Thus $A_i$, as a spectrum, is a direct sum of shifted copies of $A_0$.
On homotopy groups, the $C_p$-action cyclically permutes the variables as
indicated. Since $p$ is prime, every monomial is either fixed or
lies in an orbit of length $p$.
The fixed monomials are precisely
\[
 \prod_{j=1}^i
 \bigl(u_j\,\sigma u_j\cdots\sigma^{p-1}u_j\bigr)^{n_j},
 \qquad n_1,\ldots,n_i\geq 0.
\]
The orbits of length $p$ give summands that are induced, so that their Tate
constructions vanish (the projection formula identifies
the base change of each induced $A_{-1}$-module with an induced
$A_0$-module).

For a fixed monomial, set $d=n_1+\cdots+n_i$. Its summand is
\[
 \Sigma^{2d\rho_p}A_0,\qquad \rho_p=\mathbb R[C_p].
\]
The representation suspension records the cyclic permutation of
the $p$ copies of $\mathbb R^{2d}$. The diagonal inclusion $\mathbb R^{2d}\to(\mathbb R^{2d})^p$
induces an equivalence
$
 \Sigma^{2d}A_0^{tC_p}
 \xrightarrow{\ \simeq\ }
 \bigl(\Sigma^{2d\rho_p}A_0\bigr)^{tC_p}
$.
There are only finitely many monomials in each degree, and the
connectivity of the summands tends to infinity with the degree.
Thus the direct sum of the summands agrees with their product and the Tate construction preserves products of uniformly bounded below spectra. Thus we obtain
\begin{equation}\label{eq:graph-product}
 A_i^{tC_p}\simeq
 \prod_{n_1,\ldots,n_i\geq0}
 \Sigma^{2(n_1+\cdots+n_i)}A_0^{tC_p} .
\end{equation}
To write this more algebraically let $y_j$ denote the resulting degree-two class corresponding to
$u_j\,\sigma u_j\cdots\sigma^{p-1}u_j$. The multiplication is given by adding exponents.
Thus together with \eqref{eq:input} and the description of $C^i$ above this gives
an isomorphism of graded rings
\[
\Fp[\![z_1,\ldots,z_i]\!][x_0^{\pm1}]\xrightarrow{\;\cong\;}\pi_*(A_i^{tC_p}),
 \qquad x_0\longmapsto x_0,\quad z_j\longmapsto y_j/x_0.
\]

It remains to check that this isomorphism extends \eqref{eq:map} under the identification \eqref{identification}, i.e. that the generators really are what their name suggests.
The class $y_j$ is by construction the image of $u_j$ under
\[
 P_i\longrightarrow N(P_i)^{tC_p}\longrightarrow A_i^{tC_p},
\]
where the first map is the Tate diagonal. Indeed, apply naturality
to a representative $S^2\to P_i$: its $p$-fold tensor product defines
the fixed monomial, and the diagonal $S^2\to S^{2\rho_p}$ gives
exactly the normalization used above.

The (indexed) Tate diagonal for the covering $T_i\to T_i/C_p$
\cite[Corollary 10.3 and Example 10.4]{Lawson}, together with the B\"okstedt classes,
induces a linear map
\[
 H_1(T_i/C_p;\Fp)
 \cong \pi_2(\Fp^{\otimes(T_i/C_p)})
 \longrightarrow \pi_2(A_i^{tC_p}).
\]
If we write $a_j$ for the $j$-th original quotient circle, then its
class maps to $x_j$, since the construction restricts to the cyclotomic
Frobenius. On the attached loop $a_ja_0^{-1}$ the covering is trivial,
so the construction restricts to the ordinary Tate diagonal and its
class maps to $y_j$. By linearity we get $y_j=x_j-x_0$, which proves
the claim.
\end{proof}

\begin{lemma}\label{prop:completed-descent}
The augmented cosimplicial graded abelian group $\Fp \to C^\bullet$ is a derived limit diagram.  Equivalently,
\begin{equation}\label{eq:cohomology}
 H^s(C^\bullet)_t=
 \begin{cases}
 \Fp,&s=t=0,\\
 0,&\text{otherwise}.
 \end{cases}
\end{equation}
\end{lemma}

Without completion, the cosimplicial object $C^\bullet$ is the Amitsur complex of the faithfully
flat map $\Fp\to\Fp[x^{\pm1}]$ so it is a limit. 
Geometrically, this corresponds to the \v{C}ech
nerve of $\mathbb G_m\to\mathrm{pt}$. Completion takes the formal
neighbourhood of the diagonal in each term. We were hoping for some general formal geometry descent statement for this situation but have not been able to find one, so we give a direct proof by showing that 
degreewise completion preserves exactness by expressing it
as an inverse limit of ordinary Amitsur complexes. 

\begin{proof}
All limits below are taken  levelwise in the internal grading.
Put $R=\Fp[x^{\pm1}]$ and $R_n=\Fp[x^{\pm p^n}]\subseteq R$
for $n\geq0$, with $|x|=2$.
Following \cite[Remark 4.12]{Ogus} we
consider the Amitsur complexes
\begin{equation*}
 C_n^i:=R^{\otimes_{R_n}(i+1)}\cong
 \frac{\Fp[x_0^{\pm1},\ldots,x_i^{\pm1}]}
 {((x_1-x_0)^{p^n},\ldots,(x_i-x_0)^{p^n})}.
\end{equation*}
Since $R_n\subseteq R$ is faithfully flat, these resolve $R_n$.
The defining ideals are cofinal with the powers of $I_i$, and the
transition maps are surjective in every degree.
Thus $C^\bullet\simeq\mathrm R\!\varprojlim_n C_n^\bullet$ as cosimplicial objects, and
commuting derived limits gives
\[
 \mathrm R\!\varprojlim_{\Delta} C^\bullet
 \simeq \mathrm R\!\varprojlim_n
          \mathrm R\!\varprojlim_{\Delta} C_n^\bullet
 \simeq \mathrm R\!\varprojlim_n R_n
 \simeq \Fp.
\]
For the last equivalence, the tower $(R_n)_{2k}$ is constant
$\Fp$ for $k=0$ and eventually zero for $k\ne0$, since $p^n$
eventually does not divide $k$. All odd degrees vanish.
These equivalences identify the given augmentation with an
isomorphism, proving the claim.
\end{proof}

\begin{proof}[Proof of Theorem~\ref{thm:main}]
By the formal reduction in \cite[proof of Theorem III.1.7]{NS},
it suffices to prove that the Tate diagonal is an equivalence for
$X=\Fp$. Lemmas~\ref{prop:completion} and~\ref{prop:completed-descent}
show that
\[
 H^s\bigl(\pi_t(A_\bullet^{tC_p})\bigr)=
 \begin{cases}
 \Fp,&s=t=0,\\
 0,&\text{otherwise}.
 \end{cases}
\]

By Lemma~\ref{lem:tate-descent}, we get a conditionally convergent descent spectral sequence
\[
 E_2^{s,t}
 =H^s\bigl(\pi_t(A_\bullet^{tC_p})\bigr)
 \Rightarrow\pi_{t-s}(A_{-1}^{tC_p}).
\]
This half-plane spectral sequence collapses at $E_2$, so it
converges strongly by
\cite[Theorem 7.1]{Boardman}.
Thus $\pi_*(A_{-1}^{tC_p})\cong\Fp$, concentrated in degree zero.
Since the Tate diagonal
$\Delta_p:\Fp\to A_{-1}^{tC_p}$ is unital, it is an equivalence,
completing the proof.
\end{proof}

\subsection*{Acknowledgments.} 
The author would like to thank Thorger Gei\ss, Jeremy Hahn, Ishan Levy, Phil Pützstück and Maxime Ramzi for helpful discussions and comments on the proof.

The author was funded by the Deutsche Forschungsgemeinschaft (DFG, German Research
Foundation) -- Project-ID 427320536 -- SFB 1442, as well as under Germany’s Excellence Strategy
EXC2044/2–390685587, Mathematics M\"unster: Dynamics–Geometry–Structure.

\subsection*{AI Statement.} AI was used to improve this document: help with grammar, spelling, phrasing, to provide
referee-style feedback on the proofs and the writing, including
identifying possible gaps and suggesting improvements, find references and generate the figures.

\end{document}